\documentclass[11pt]{amsart}
\usepackage{}

\usepackage{amscd}
\usepackage{amssymb}
\usepackage{amsmath}
\usepackage{amsfonts}
\usepackage{amssymb, latexsym, amsmath, pb-diagram}
\usepackage[all,cmtip]{xy}
\usepackage{enumerate}
\usepackage{graphicx}
\usepackage{epic}
\usepackage{overpic}
\usepackage{array}
\usepackage{subfig}
\usepackage{caption}
\usepackage{varioref}

\DeclareCaptionSubType[Alph]{figure}
\def\w{\widetilde}
\def\wh{\widehat}

\newtheorem{thm}{\indent Theorem}[section]
\newtheorem{prop}[thm]{\indent Proposition}
\newtheorem{lem}[thm]{\indent Lemma}

\theoremstyle{definition}

\theoremstyle{remark}

\newtheorem{Def}[thm]{\indent \rm DEFINITION}
\newtheorem{conj}[thm]{\indent \rm CONJECTURE}

\newtheorem{cons}[thm]{\indent \rm CONSTRUCTION}
\newtheorem{rem}[thm]{\indent \rm REMARK}

\newtheorem*{ass}{\indent \rm ASSERTION}
\newtheorem{que}{\indent \rm QUESTION}[]

\numberwithin{equation}{section}
\numberwithin{table}{section}

\begin{document}
\title{Moment-angle manifolds and connected sums of sphere products}
\author[F.~Fan, L.~Chen, J.~Ma \& X.~Wang]{Feifei Fan, Liman Chen, Jun Ma and Xiangjun Wang}
\thanks{The authors are supported by NSFC grant No. 11261062 and SRFDP No.20120031110025}
\address{Feifei Fan, School of Mathematical Sciences and LPMC, Nankai University, Tianjin 300071, P.~R.~China}
\email{fanfeifei@mail.nankai.edu.cn}
\address{Liman Chen, School of Mathematical Sciences and LPMC, Nankai University, Tianjin 300071, P.~R.~China}
\email{chenlimanstar1@163.com}
\address{Jun Ma, School of Mathematical Sciences, Fudan University, Shanghai 200433, P.~R.~China}
\email{tonglunlun@gmail.com}
\address{Xiangjun Wang, School of Mathematical Sciences and LPMC, Nankai University, Tianjin 300071, P.~R.~China}
\email{xjwang@nankai.edu.cn}
\subjclass[2010]{Primary 13F55, 14M25, 55U10; Secondary 57R18, 57R19}
\date{}
\maketitle
\begin{abstract}
Corresponding to every finite simplicial complex $K$, there is a
¡°moment-angle¡± complex $\mathcal {Z}_{K}$; if $K$ is a triangulation of a sphere, $\mathcal {Z}_{K}$ is a compact manifold.
The question of whether $\mathcal {Z}_{K}$ is a connected sum of sphere products was considered in {\cite[\S 11]{BM06}}.
So far, all known examples of moment-angle manifolds which are homeomorphic to connected sums of sphere products have the property that every product is of exactly two spheres.
In this paper, we give a example whose cohomology ring is isomorphic to that of a connected sum of sphere products with one product of three spheres. We also give some general properties of this kind of moment-angle manifolds.
\end{abstract}

\section{Introduction}
Throughout this paper, we assume that $m$ is a positive integer and $[m]=\{1,2,\dots,m\}$.
For an abstract simplicial complex $K$ with $m$ vertices labeled by $[m]$ and a sequence
$I=(i_1,\dots,i_k)\subseteq[m]$ with $1\leq i_1\leq\cdots\leq i_k\leq m$,
we denote by $K_I$ the \emph{full subcomplex} of $K$ on $I$,
and $\wh I=[m]\setminus I$.
\subsection{Moment-angle complex}
Given a simple polytope $P$ with $m$ fecets, Davis and Januszkiewicz \cite{DJ91} constructed a manifold  $\mathcal{Z}_P$ with an action of a real torus $T^m$. After that Buchstaber and Panov  \cite{BP00} generalized this definition to any simplicial complex $K$, that is
\[\mathcal {Z}_{K}=\bigcup_{\sigma\in K}(D^2)^{\sigma}\times (S^1)^{[m]\setminus \sigma},\]
and named it the \emph{moment-angle complex} associated to $K$, whose study connects algebraic geometry, topology, combinatorics, and commutative algebra.
This cellular complex is always $2$-connected and
has dimension $m+n+1$, where $n$ is the dimension of $K$.

It turns out that the algebraic topology of a moment-angle complex $\mathcal {Z}_{K}$,
such as the cohomology ring and the homotopy groups is intimately related to the combinatorics of the underlying simplicial complex $K$.

\subsection{Moment-angle manifold}
Now suppose that $K$ is an $n$-dimentional simplicial sphere (a triangulation of a sphere) with $m$ vertices. Then, as shown by Buchstaber and Panov \cite{BP00},
the moment-angle complex $\mathcal {Z}_{K}$ is a manifold of dimension $n+m+1$, referred to as a \emph{momnet-angle manifold}.
In particular, if $K$ is a polytopal sphere (see Definition \ref{def:1}), or more generally a starshaped sphere (see Definition \ref{def:0}), then $\mathcal {Z}_{K}$ admits a smooth structure.
\begin{Def}\label{def:1}
A \emph{polytopal sphere} is a triangulated sphere isomorphic to the boundary complex of a simplicial polytope.
\end{Def}

\begin{Def}\label{def:0}
a simplicial sphere $K$ of dimension $n$ is said to be \emph{sastarshaped} if there is a geometric realization $|K|$ of $K$ in $\mathbb{R}^n$ and a point $p\in \mathbb{R}^n$
with the property that each ray emanating from $p$ meets $|K|$ in exactly one point.
\end{Def}

\begin{rem}
A polytopal sphere is apparently a starshaped sphere, but for $n\geq 3$, there are examples that are starshaped and not polytopal. The easiest such example is given by the \emph{Br\"uckner sphere} (see \cite{GS67}).
\end{rem}

The topology of a moment-angle manifold can be quite complicated.
The complexity increases when the dimension $n$ of the associated simplicial sphere $K$ increases.
for $n=0$, $\mathcal {Z}_{K}$ is $S^3$. For $n=1$, $K$ is the boundary of a polygon,
and $\mathcal {Z}_{K}$ is a connected sum of sphere products. In higher dimensions, the situation becomes much more complicated.
On the other hand, McGavran \cite{M79} showed that, for any $n>0$, there are infinitely
many $n$-dimensional polytopal spheres whose corresponding moment-angle manifolds are connected sums of sphere products.

\begin{thm}[McGavran, see {\cite[Theorem 6.3]{BM06}}]\label{thm:2}
Let $K$ be a polytopal sphere dual to the simple polytope obtained from the $k$-simplex by cutting off vertices for $l$ times.
Then the corresponding moment-angle manifold is homeomorphic to a connected sum of sphere products
\[\mathcal {Z}_{K}\cong \overset{l}{\underset{j=1}\#}j\binom{l+1}{j+1}S^{j+2}\times S^{2k+l-j-1}.\]
\end{thm}

For $k=2$ or $3$, the above theorem gives all moment-angle manifolds which are homomorphic to connected sums of sphere products
(see {\cite[Proposition 11.6]{BM06}}).
Nevertheless, in higher dimension they are not the only ones whose cohomology ring is isomorphic
to that of a connected sum of sphere products. Bosio and Meersseman {\cite[\S 11]{BM06}} gave many other examples of moment-angle manifolds
whose cohomology rings have this property.
We notice that all examples of connected sums of sphere products given in \cite{BM06} have the property that every product is of two spheres, this leads to a question:
\begin{que}
If $\mathcal {Z}_{K}$ is a connected sum of sphere products, is it ture that every product is of exactly two spheres?
\end{que}
In this paper (Proposition \ref{prop:1}),
we give a negative answer to this question at the aspect of cohomology rings, by constructing a $3$-dimentional polytopal
sphere, so that the cohomology ring of the corresponding moment-angle manifold is isomorphic to the cohomology ring of the connected sum of sphere products
\[S^3\times S^3\times S^6\#(8)S^5\times S^7\#(8)S^6\times S^6.\]

\section{cohomology ring of moment-angle complex}\label{sec:2}
\begin{Def}\label{def:2}
Let $K$ be a simplicial complex with vertex set $[m]$.
A \emph{missing face} of $K$ is a sequence $(i_1,\dots,i_k)\subseteq[m]$ such that $(i_1,\dots,i_k)\not\in K$,
but every proper subsequence of $(i_1,\dots,i_k)$ is a simplex of $K$. Denote by $MF(K)$ the set of all missing faces of $K$.
\end{Def}

From definition \ref{def:2}, it is easy to see that if $K_I$ is a full subcomplex of $K$, then $MF(K_I)$ is a subset of $MF(K)$. Concretely,
\[MF(K_I)=\{\sigma\in MF(K):\sigma\subseteq I\}.\]

Let $R[m]=R[v_1,\dots,v_m]$ denote the graded polynomial algebra over $R$, where $R$ is a field or $\mathbb{Z}$, $\mathrm{deg}v_i=2$.
The \emph{face ring} (also known as the \emph{Stanley-Reisner ring}) of a simplicial complex $K$ on the vertex set [m] is the quotient ring
\[R(K)=R[m]/\mathcal {I}_K,\]
where $\mathcal {I}_K$ is the ideal generated by all square free monomials $v_{i_1}v_{i_2}\cdots v_{i_s}$
such that $(i_1,\dots,i_s)\in MF(K)$.

The following result is used to calculate the cohomology ring of $\mathcal {Z}_{K}$,
which is proved by Buchstaber and Panov \cite[Theorems 7.6]{BP02} for the case over a field, \cite{BBP04} for the general case;
see also \cite[Theorem 4.7]{P08}. Another proof of Theorem \ref{thm:1} for the case over $\mathbb{Z}$ was given by Franz \cite{M06}.

\begin{thm}[Buchstaber-Panov, {\cite[Theorem 4.7]{P08}}]\label{thm:1}
Let $K$ be a abstract simplicial complex with $m$ vertices. Then the cohomology
ring of the moment-angle complex $\mathcal {Z}_{K}$ is given by the isomorphisms
\begin{align*}
H^*(\mathcal {Z}_{K};R)\cong \mathrm{Tor}_{R[m]}^{*,*}(R(K),R)\cong\bigoplus_{I\subseteq [m]} \w {H}^*(K_I;R)
\end{align*}
where
\[H^p(\mathcal {Z}_{K};R)\cong \bigoplus_{\substack{J\subseteq[m]\\-i+2|J|=p}}\mathrm{Tor}^{-i,\,2|J|}_{R[m]}(R(K),R)\]
and
\[\mathrm{Tor}^{-i,\,2|J|}_{R[m]}(R(K),R)\cong \w {H}^{|J|-i-1}(K_J;R).\]
\end{thm}

\begin{rem}\label{rem:1}
There is a canonical ring structure on $\bigoplus_{I\subseteq [m]} \w {H}^*(K_I)$
(called the \emph{Hochster ring} and denoted by $\mathcal {H}^{*,*}(K)$, where $\mathcal {H}^{i,J}(K)=\w H^i(K_J)$) given by the maps
\[\eta:\w H^{p-1}(K_I)\otimes H^{q-1}(K_J)\to H^{p+q-1}(K_{I\cup J}),\]
which are induced by the canonical simplicial inclusions $K_{I\cup J}\to K_I*K_J$ (join
of simplicial complexes) for $I\cap J=\emptyset$ and zero otherwise.
Precisely, Let $\w C^q(K)$ be the $q$th reduced simplicial cochain group of $K$.
For a oriented simplex $\sigma=(i_1,\dots,i_p)$ of $K$ (the orientation is given by the order of vertices of $\sigma$), denote by
$\sigma^*\in\w C^{p-1}(K)$ the basis cochain corresponding to $\sigma$; it takes value $1$ on $\sigma$ and vanishes on all other simplices.
Then for $I,J\in[m]$ with $I\cap J=\emptyset$, we have isomorphisms of reduced simplicial cochains
\begin{align*}
\mu :\w C^{p-1}(K_I)\otimes \w C^{q-1}(K_J)&\to \w C^{p+q-1}(K_I*K_J), \quad p,q\geq 0\\
\sigma^*\otimes\tau^*&\mapsto (\sigma\sqcup \tau)^*
\end{align*}
where $\sigma\sqcup \tau$ means the juxtaposition of $\sigma$ and $\tau$.
Given two cohomology classes $[c_1]\in \w H^{p-1}(K_I)$ and $[c_2]\in\w H^{q-1}(K_J)$, which are represented by the cocycles
$\sum_i \sigma_i^*$ and $\sum_j\tau_j^*$ respectively. Then
\[\eta([c_1]\otimes[c_2])=\varphi^*([\mu(\sum_{i,j}\sigma_i^*\otimes\tau^*_j)]),\]
where $\varphi:K_{I\cup J}\to K_I*K_J$ is the simplicial inclusion.

We denote by $\psi([c])$ the inverse image of a class $[c]\in \bigoplus_{I\subseteq [m]} \w {H}^*(K_I)$
by the composition of the two isomorphisms in Theorem \ref{thm:1}.
Given two cohomology classes $[c_1]\in \w H^{p}(K_I)$ and $[c_2]\in\w H^{q}(K_J)$, define
\[[c_1]*[c_2]=\eta([c_1]\otimes[c_2]).\]
Bosio and Meersseman proved in \cite{BM06} (see also \cite[Proposition 3.2.10]{BP13}) that, up to sign
\[\psi([c_1])\smile \psi([c_2])=\psi([c_1]*[c_2]).\]
\end{rem}

\begin{rem}\label{rem:2}
Baskakov showed in \cite{B02} (see {\cite[Theorem 5.1]{P08}}) that the isomorphisms in Theorem \ref{thm:1} are functorial with respect to simplicial maps (here we only consider simplicial inclusions). That is, for a simplicial inclusion $i:K'\hookrightarrow  K$
(suppose the vertex sets of $K'$ and $K$ are $[m']$ and $[m]$ respectively) which
induces natural inclusions
\[\phi: \mathcal{Z}_{K'}\hookrightarrow  \mathcal {Z}_{K}\]
and
\[i|_{K'_I}: K'_I\hookrightarrow  K_I,\ \text{ for each } I\subseteq [m'],\]
there is a commutative diagram of algebraic homomorphisms
\[\begin{CD}
H^*(\mathcal{Z}_{K})@>\phi^*  >>H^*(\mathcal {Z}_{K'})\\
@V\cong VV @VV\cong V\\
\underset{I\subseteq [m]}\bigoplus \w {H}^*(K_I)@>\bigoplus_I(i|_{K'_I})^* >>\underset{I\subseteq [m']}\bigoplus\w {H}^*(K'_I)
\end{CD}\]
\end{rem}
Actually, there are three ways to calculate the integral cohomology ring of a moment-angle complex $\mathcal {Z}_{K}$.

(1) The first is to calculate the Hochster ring $\mathcal {H}^{*,*}(K)$ of $K$ and apply the isomorphisms in Theorem \ref{thm:1}.

(2) The second is to calculate $\mathrm{Tor}^{*,*}_{\mathbb{Z}[m]}(\mathbb{Z}(K),\mathbb{Z})$ by means of the Koszul resolution
(\cite[Theorem 7.6 abd Theorem 7.7]{BP02}), that is
\[\mathrm{Tor}_{\mathbb{Z}[m]}^{*,*}(\mathbb{Z}(K),\mathbb{Z})\cong H(\Lambda[u_1,\dots,u_m]\otimes \mathbb{Z}(K),d),\]
where $\Lambda[u_1,\dots,u_m]$ is the exterior algebra over $\mathbb{Z}$ generated by $m$ generators. On the right side, we have
\[\mathrm{bideg}u_i=(-1,2),\quad \mathrm{bideg}v_i=(0,2),\quad du_i=v_i,\quad dv_i=0.\]
In fact, there is a simpler way to calculate the cohomology of this differential graded algebra by applying the following result
\begin{prop}[{\cite[Lemma 3.2.6]{BP13}}]
The projection homomorphism
\[\varrho: \Lambda[u_1,\dots,u_m]\otimes \mathbb{Z}(K)\to A(K)\]
induces an isomorphism in cohomology, where $A(K)$ is the quotient algebra
\[A(K)=\Lambda[u_1,\dots,u_m]\otimes \mathbb{Z}(K)/(v_i^2=u_iv_i=0,\ 1\leq i\leq m).\]
\end{prop}

(3) The third way is to use the Taylor resolution for $\mathbb{Z}(K)$ to calculate $\mathrm{Tor}_{\mathbb{Z}[m]}(\mathbb{Z}(K),\mathbb{Z})$.
This was introduced first by Yuzvinsky in \cite{Y99}. Wang and Zheng \cite{WZ13} applied this method to toric topology. Concretely, let
$\mathbb{P}=MF(K)$, and let $\Lambda[\mathbb{P}]$ be the exterior algebra generated by $\mathbb{P}$.
Given a monomial $u=\sigma_{k_1}\sigma_{k_2}\cdots\sigma_{k_r}$ in $\Lambda[\mathbb{P}]$, let
\[S_u=\sigma_{k_1}\cup\sigma_{k_2}\cdots\cup\sigma_{k_r}.\]
Define $\mathrm{bideg}\,u=(-r,2|S_u|)$, and define
\[\partial_i(u)=\sigma_{k_1}\cdots\wh{\sigma_{k_i}}\cdots\sigma_{k_r}=\sigma_{k_1}\cdots\sigma_{k_{i-1}}\sigma_{k_{i+1}}\cdots\sigma_{k_r}.\]

Let $(\Lambda^{*,*}[\mathbb{P}],d)$ be the cochain complex (with a different product structure from $\Lambda[\mathbb{P}]$) induced from the bi-graded exterior algebra on $\mathbb{P}$.
The differential $d: \Lambda^{-q,*}[\mathbb{P}]\to \Lambda^{-(q-1),*}[\mathbb{P}]$ is given by
\[d(u)=\sum_{i=1}^q (-1)^i\partial_i(u)\delta_i,\]
where $\delta_i=1$ if $S_u=S_{\partial_i(u)}$ and zero otherwise. The product structure in $(\Lambda^{*,*}[\mathbb{P}],d)$ is given by
\[u\times v=
\begin{cases}
u\cdot v \quad &\text{if } S_u\cap S_v=\emptyset,\\
0 &\text{otherwise,}
\end{cases}\]
where {\LARGE$\cdot$} denote the ordinary product in the exterior algebra $\Lambda[\mathbb{P}]$.
\begin{prop}[see {\cite[Theorem 2.6 and Theorem 3.2]{WZ13}}]
There is a algebraic isomorphism
\[\mathrm{Tor}_{\mathbb{Z}[m]}^{*,*}(\mathbb{Z}(K),\mathbb{Z})\cong H(\Lambda^{*,*}[\mathbb{P}],d)\]
\end{prop}

\section{Construction of a polytopal $3$-sphere with eight vertices}
In this section, we construct a $3$-dimensional polytopal sphere $K$ with eight vertices,
such that the cohomology ring of the correponding moment-angle manifold $\mathcal {Z}_{K}$
is isomorphic to the the cohomology ring of a connected sum of sphere products with one product of three spheres.
\begin{cons}\label{cons:1}
We construct $K$ by three steps. First give a $2$-dimensional simplicial complex $K_0$ with $4$ vertices shown in Figure \ref{fig:1}.
\[MF(K_0)=\{(1,2,3), (1,3,4)\}.\] It has two subcomplex
$K_1$ and $K_2$ also shown in Figure \ref{fig:1}. Next let $L_1=K_0\cup \mathrm{cone}(K_1)$ with a new vertex $5$.
(i.e., $L_1$ is the mapping cone of the inclusion map $K_1\hookrightarrow K_0$), and let $L_2=K_0\cup \mathrm{cone}(K_2)$ with a new vertex $6$.
Let $K_0'=L_1\cup L_2$ be a simplicial complex obtained by gluing $L_1$ and $L_2$ along $K_0$ (see Figure \ref{fig:2}). Then
\[MF(K_0')=\{(1,2,3), (1,3,4), (2,3,5), (3,4,6), (5,6)\}.\]
Note that $K_0'$ can be viewed as a "thick" $2$-sphere with two $3$-simplices $(1,2,4,5)$ and $(1,2,4,6)$,
shown shaded in Figure \ref{fig:2}.
$K_0'$ has two subcomplexes $K_1'$ and $K_2'$ (see Figure \ref{fig:2}),
which are all triangulations of $S^2$. Let $\mathrm{cone}(K_2')$ be the cone of $K_2'$ with a new vertex $8$.
Then it is easy to see that $K'=K_0'\cup \mathrm{cone}(K_2')$ is a triangulation of $D^3$ and its boundary
is $K_1'$. Finally, let $K=K'\cup \mathrm{cone}(K_1')$ with a new vertex $7$.
Clearly, $K$ is a triangulation of $S^3$, and the missing faces of $K$ are
\begin{equation}\label{eq:0}
\begin{split}
MF(K)=\{&(1,2,3),(1,3,4),(2,3,5),(3,4,6),(5,6),\\
&(1,4,7),(4,6,7),(1,2,8),(2,5,8),(7,8)\}.
\end{split}
\end{equation}
\begin{figure}[!ht]
\captionsetup[subfloat]{labelformat=empty,textfont=normalsize}
\centering
\subfloat[$K_0$]{\includegraphics[scale=0.35]{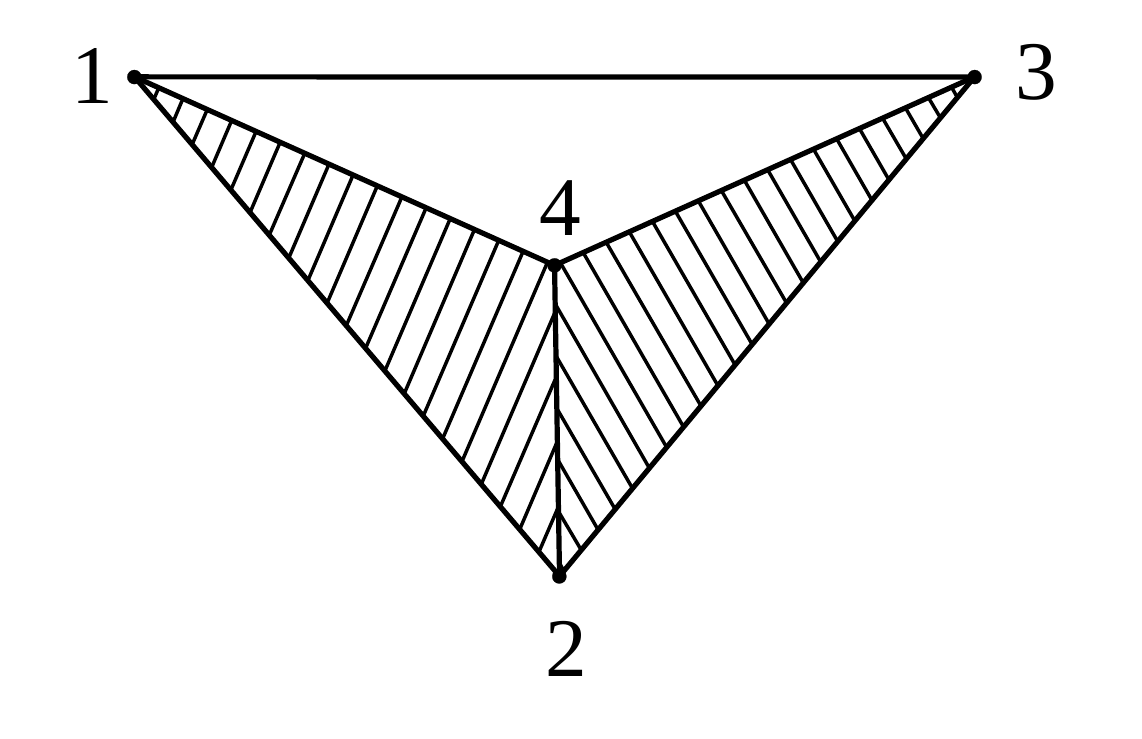}}
\subfloat[$K_1$]{\includegraphics[scale=0.35]{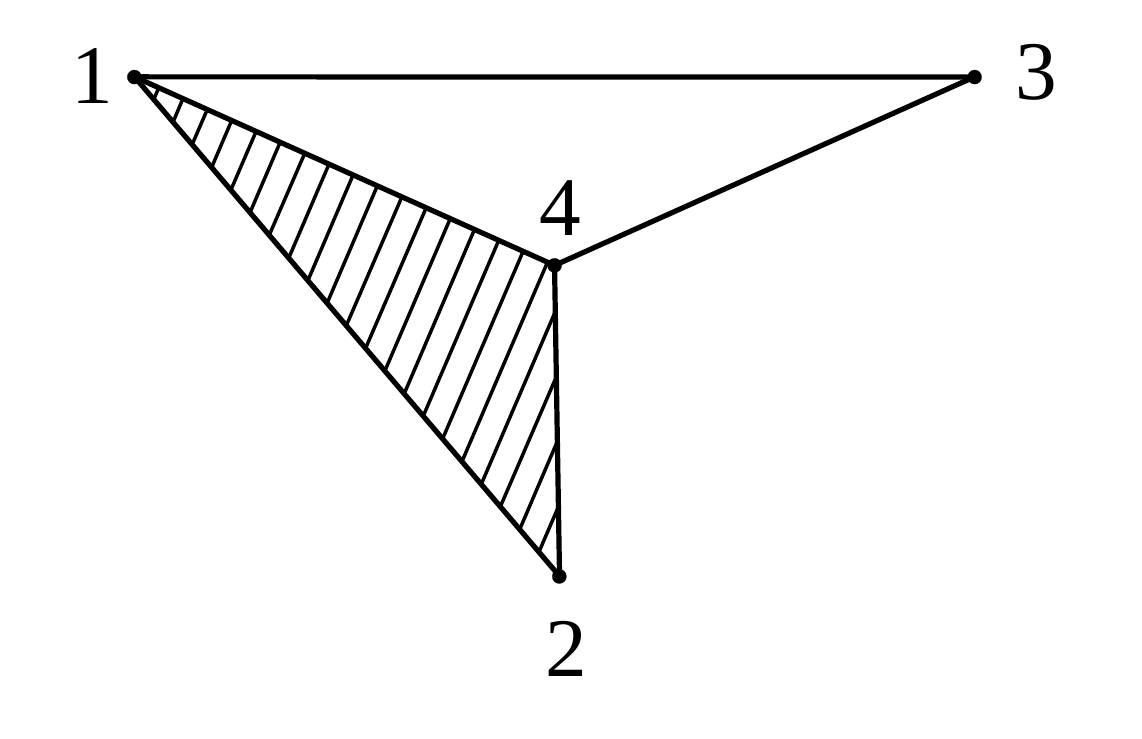}}
\subfloat[$K_2$]{\includegraphics[scale=0.35]{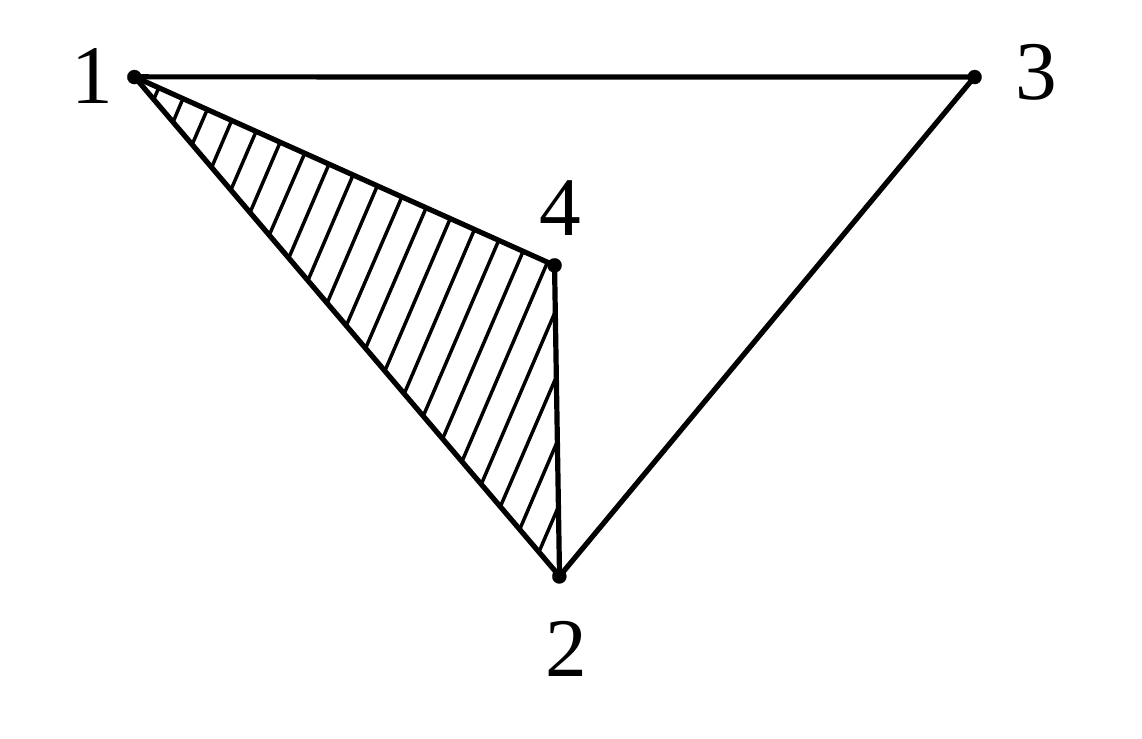}}
\caption{$K_0$, $K_1$ and $K_2$.}\label{fig:1}
\end{figure}
\begin{figure}[!ht]
\captionsetup[subfloat]{labelformat=empty,textfont=normalsize}
\centering
\subfloat[$K_0'$]{\includegraphics[scale=0.4]{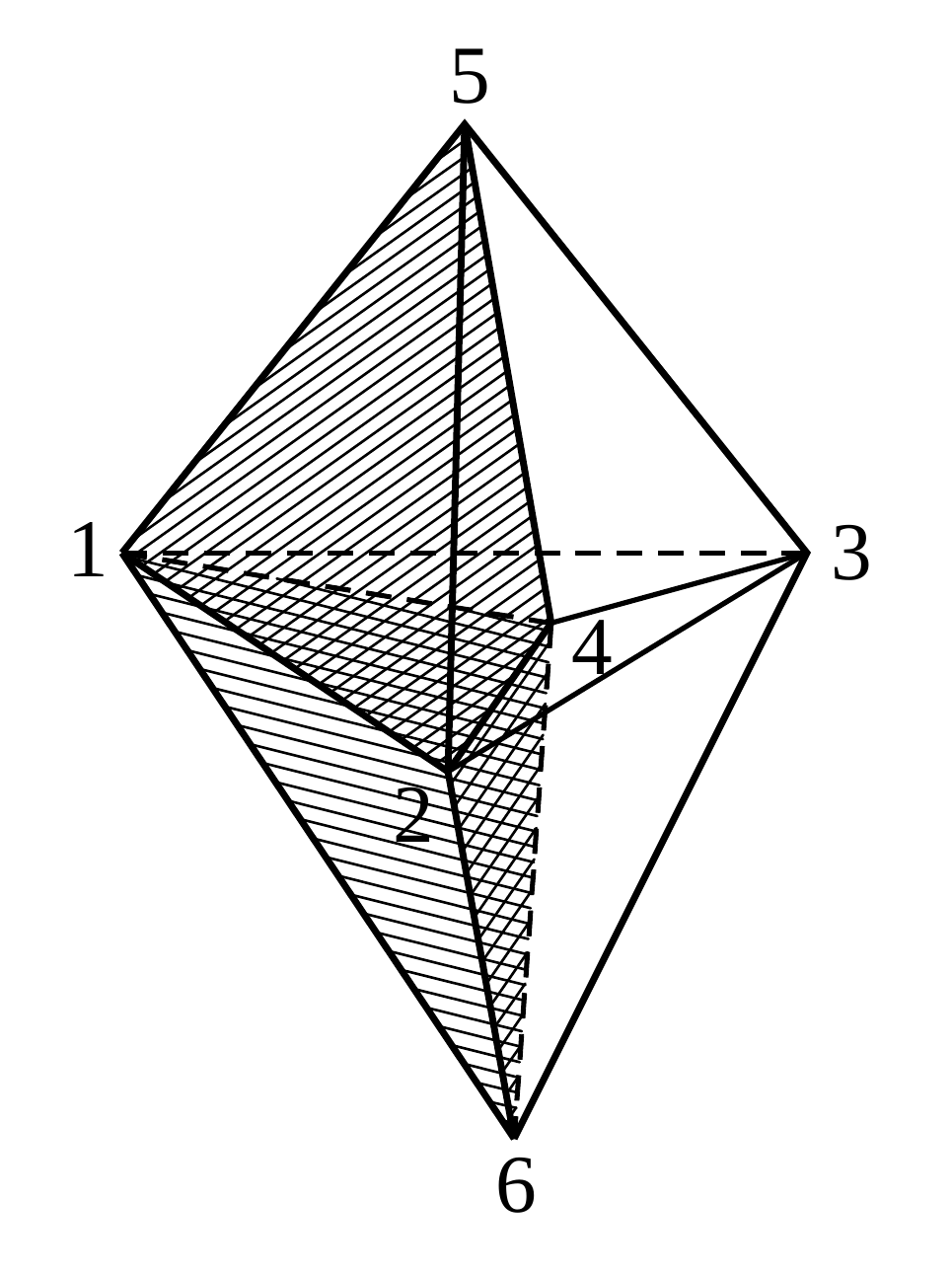}}
\subfloat[$K'_1$]{\includegraphics[scale=0.4]{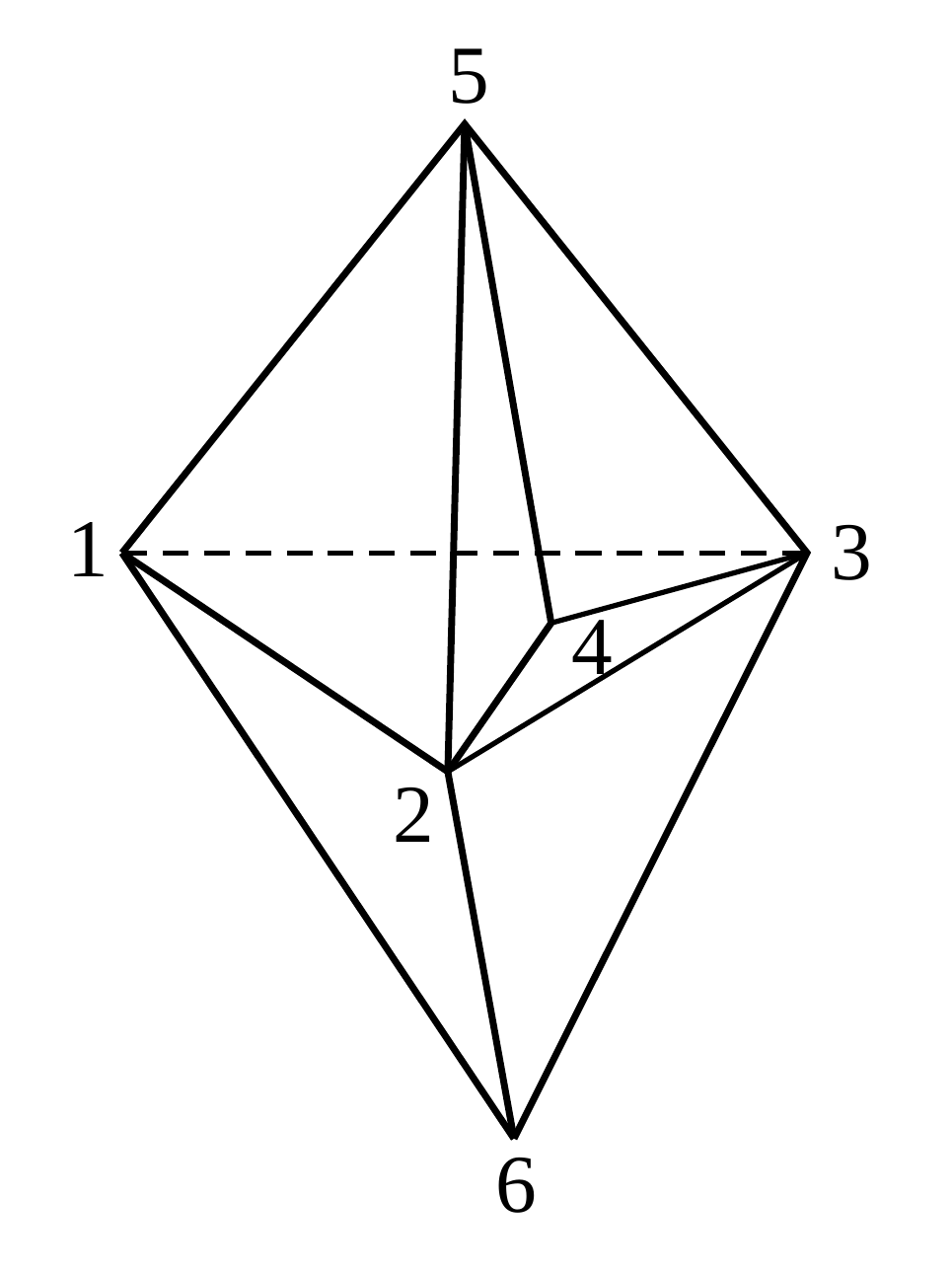}}
\subfloat[$K'_2$]{\includegraphics[scale=0.4]{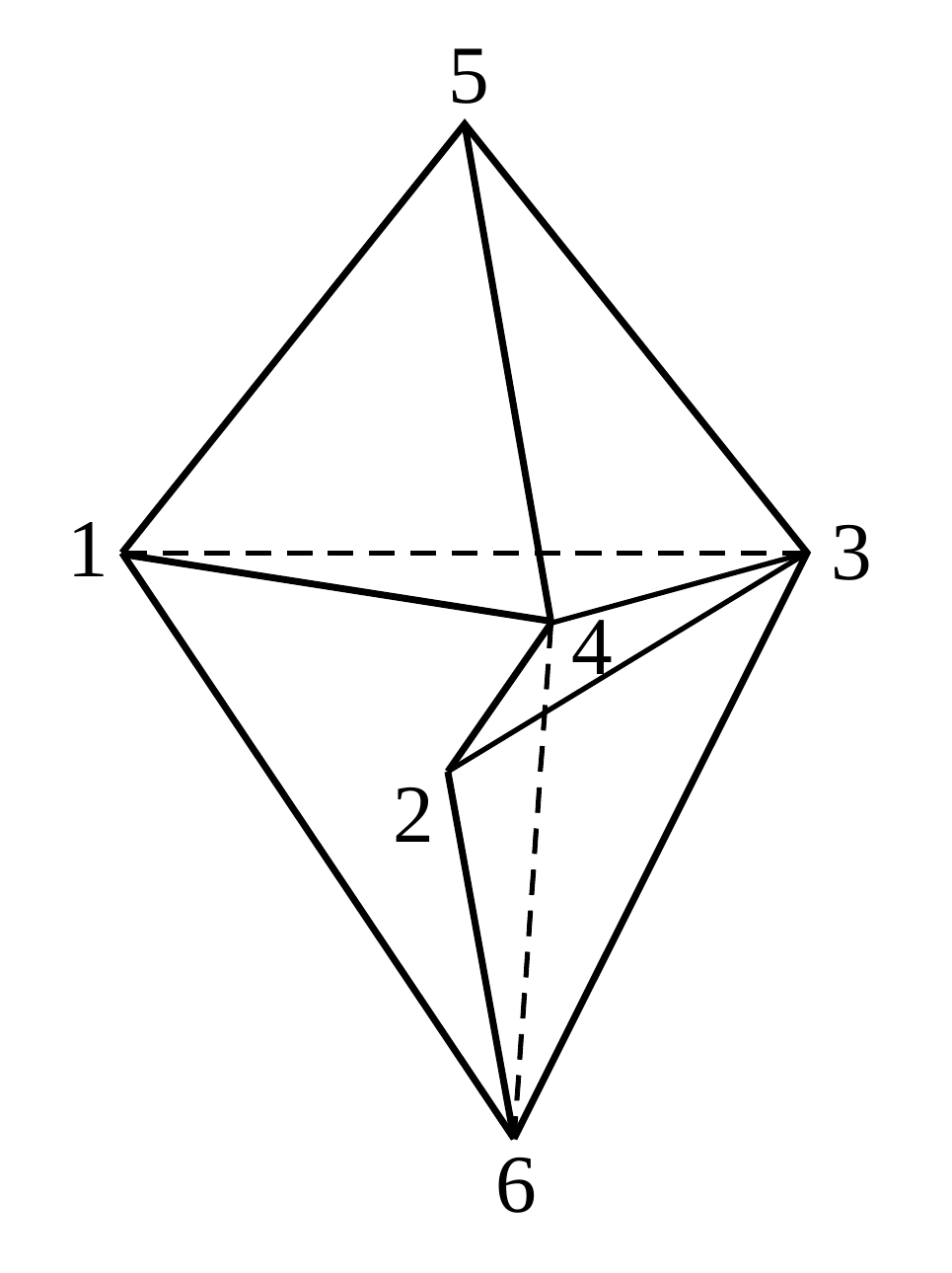}}
\caption{$K_0'$, $K'_1$ and $K'_2$.}\label{fig:2}
\end{figure}
\end{cons}

Gr\"unbaum and Sreedharan \cite{GS67}
gave a complete enumeration of the simplicial 4-polytopes with 8 vertices. A direct verification shows that $K$ we construct above
is isomorphic to the boundary of $P_{28}^8$ (a 4-polytope with $18$ facets) in \cite{GS67}.
Then $K$ is actually a polytopal sphere. From the construction we know that all $3$-simplices of $K$ are
\[\begin{split}
&(1,2,4,5),\ (1,2,4,6),\ (1,2,5,7),\ (1,2,6,7),\ (1,3,5,7),\ (1,3,6,7)\\
&(2,3,4,7),\ (2,3,6,7),\ (2,4,5,7),\ (3,4,5,7),\ (1,4,5,8),\ (1,4,6,8)\\
&(1,3,5,8),\ (1,3,6,8),\ (2,3,4,8),\ (2,3,6,8),\ (2,4,6,8),\ (3,4,5,8).
\end{split}\]

\section{connected sums of sphere products}
In the first part of this section, we calculate the cohomology ring of $\mathcal {Z}_K$ corresponding to the polytopal sphere $K$
constructed in the last section. In the second part, we give some general properties for the moment-angle manifolds whose cohomology ring is isomorphic to
that of a connected sum of sphere product.
\labelformat{subfigure}{#1}
\begin{figure}
\captionsetup[subfloat]{labelformat=simple,textfont=normalsize,labelfont=normalsize}
\centering
\subfloat[\label{subf:1}]{\includegraphics[scale=0.22]{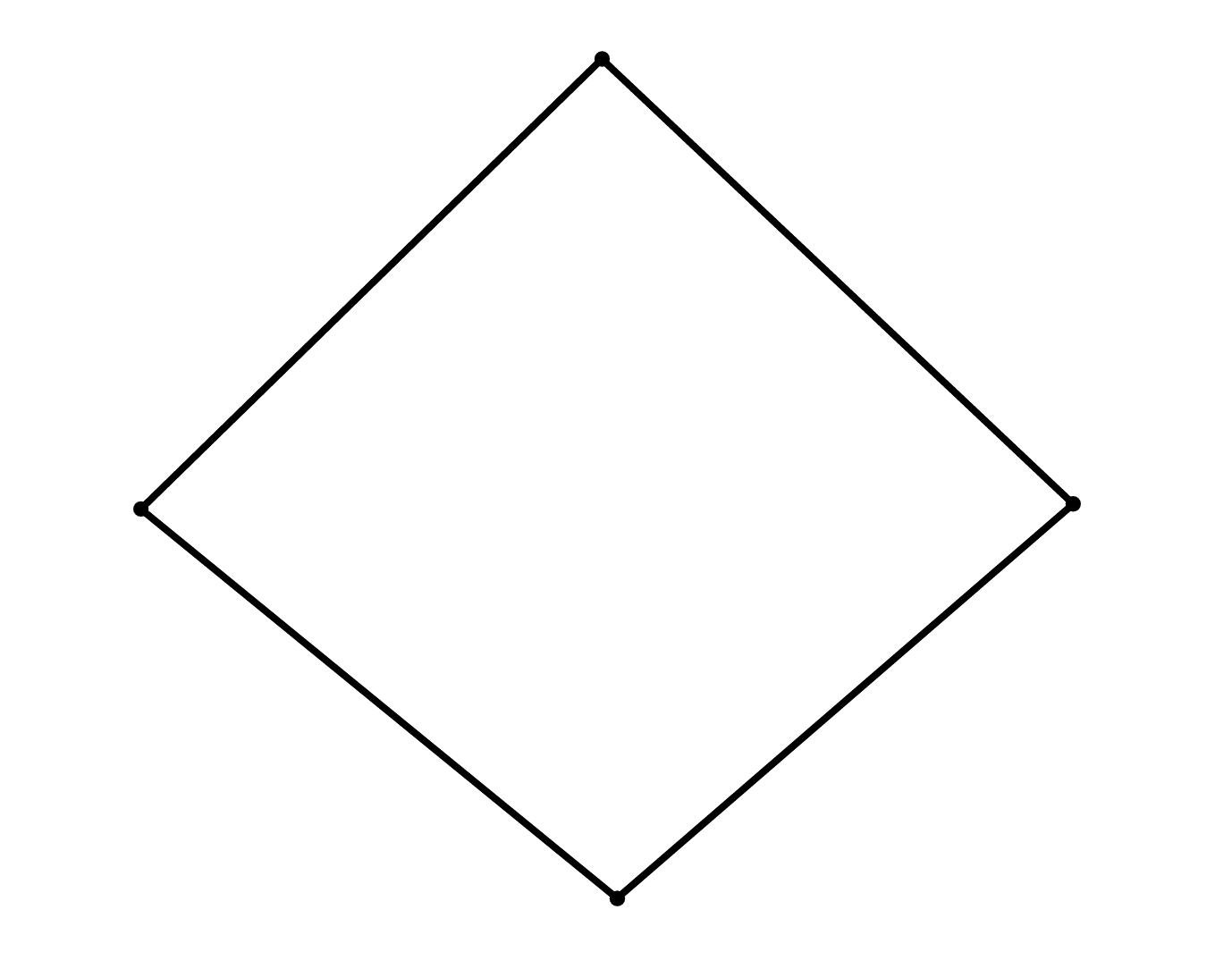}}\
\subfloat[\label{subf:4}]{\includegraphics[scale=0.32]{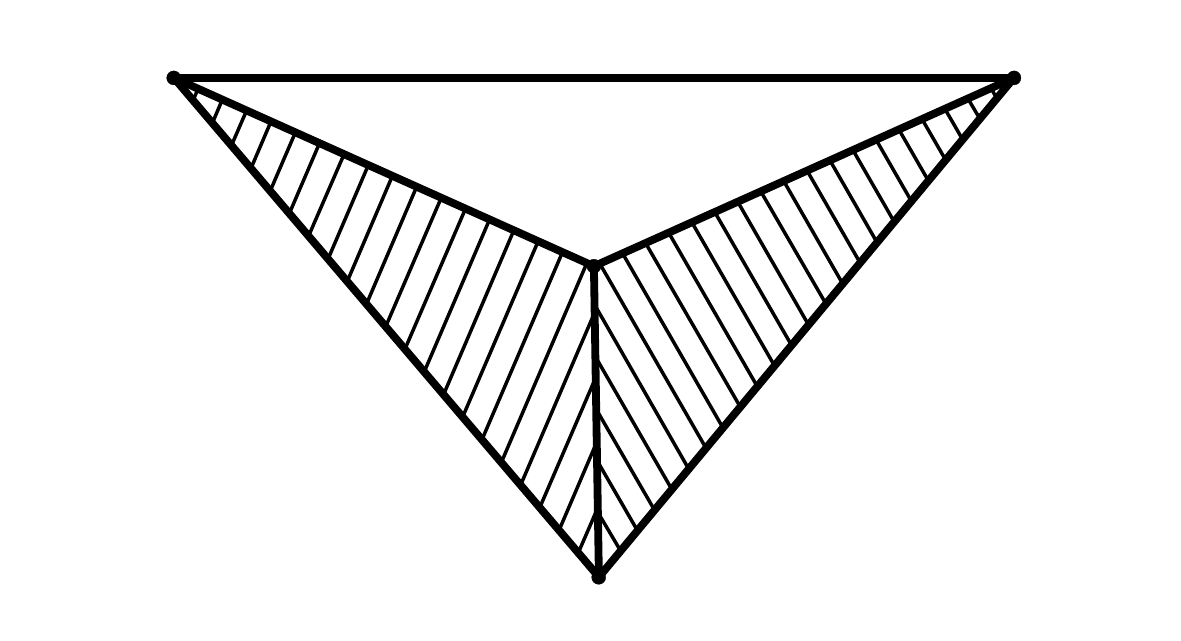}}\
\subfloat[\label{subf:5}]{\includegraphics[scale=0.32]{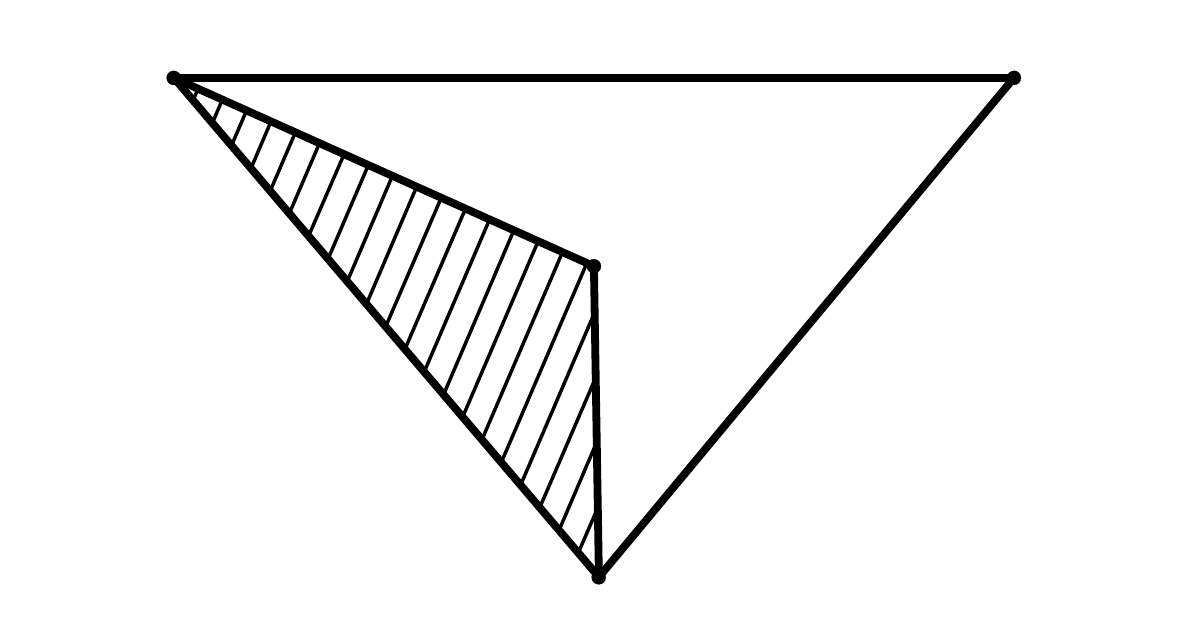}}
\caption{}\label{fig:3}
\end{figure}

\begin{prop}\label{prop:1}
For the polytopal sphere $K$ defined in Construction \ref{cons:1}, the cohomology ring of the corresponding moment-angle manifold
$\mathcal {Z}_{K}$ is isomorphic to the cohomology ring of \[S^3\times S^3\times S^6\#(8)S^5\times S^7\#(8)S^6\times S^6\].
\end{prop}

We will calculate $H^*(\mathcal {Z}_{K})$ in the first way introduced in section \ref{sec:2}. Therefore we need first to calculate the reduced cohomology rings of
all full subcomplexes of $K$. Note first the following obvious fact: Let $\Gamma$ be a simplicial complex with vertex set $[m]$. Define
$\mathcal {I}=\bigcup_{\sigma\in MF(\Gamma)}\sigma.$
If $\mathcal {I}\neq [m]$, then $\Gamma=K_{\mathcal {I}}*\Delta^{m-|\mathcal {I}|-1}$, and therefore $\Gamma$ is contractible.

Now we do this work in $6$ cases according to the cardinality of $I$ for $K_I$.
\begin{enumerate}[(1)]
\item Since the case $|I|=1$ is trivial, we start with the case $|I|=2$.
In this case, from \eqref{eq:0}, it is easy to see that $\w {H}^*(K_I)\neq 0$ if and only if $I=(5,6)$ or $(7,8)$,
and if so, $\w {H}^*(K_I)\cong \w {H}^0(K_I)\cong \mathbb{Z}$. Denote by $a_1$ (respectively $a_2$) a generator of $\w {H}^0(K_I)$
for $I=(5,6)$ (respectively $(7,8)$).

\item \label{case:2} $|I|=3$. It is easy to see that the union of any two missing faces of $K$ contains at least four vertices.
Combining the preceding argument we have that $\w {H}^*(K_I)\neq 0$
if and only if $I$ is one of the eight missing faces with three vertices in $MF(K)$,
and if so, $\w {H}^*(K_I)\cong \w {H}^1(K_I)\cong \mathbb{Z}$, whose generator we denote by $b_i$ ($1\leq i\leq 8$).

\item $|I|=4$. An easy observtion shows that the union of any three missing faces of $K$ contains at least five vertices,
and $K$ has no missing face with four vertices.
So if $K_I$ is not contractible, then it has exactly two missing faces.
Thus from \eqref{eq:0}, the form of $MF(K_I)$ is one of $\{(v_1,v_2),(v_3,v_4)\}$, $\{(v_1,v_2,v_3),(v_2,v_3,v_4)\}$ and $\{(v_1,v_2),(v_1,v_3,v_4)\}$ , for which
the corresponding simplicial complexes are respectively $A$, $B$ and $C$ shown in Figure \ref{fig:3}.
It is easy to see that they are all homotopic to $S^1$. In Table \ref{tab:1} we list all non-contractible full subcomplexes $K_I$ of $K$ for $|I|=4$
(each $I_j$ contains vertex 1). Denote by $\alpha_j$ (respectively $\alpha_j'$) a generator of
$\w {H}^*(K_{I_j})\cong \w {H}^1(K_{I_j})\cong \mathbb{Z}$ (respectively $\w {H}^*(K_{\wh{I_j}})$) for $0\leq j\leq 8$.

\begin{table}\renewcommand{\arraystretch}{1.5}
\begin{tabular}{c|c|c|c|c|c|c}\hline
&$K_{I_0}$&$K_{\wh{I_0}}$&$K_{I_1}$&$K_{\wh{I_1}}$&$K_{I_2}$&$K_{\wh{I_2}}$\\ \hline
vertex set&$\{1,2,3,4\}$&$\{5,6,7,8\}$&$\{1,2,3,5\}$&$\{4,6,7,8\}$&$\{1,2,3,8\}$&$\{4,5,6,7\}$\\ \hline
\shortstack{missing\\faces}&$\shortstack{(1,2,3)\\(1,3,4)}$&$\shortstack{(5,6)\\(7,8)}$&$\shortstack{(1,2,3)\\(2,3,5)}$
&$\shortstack{(4,6,7)\\(7,8)}$&$\shortstack{(1,2,3)\\(1,2,8)}$&$\shortstack{(4,6,7)\\(5,6)}$\\ \hline
\end{tabular}\vspace{10pt}
\begin{tabular}{c|c|c|c|c|c}\hline
$K_{I_3}$&$K_{\wh{I_3}}$&$K_{I_4}$&$K_{\wh{I_4}}$&$K_{I_5}$&$K_{\wh{I_5}}$\\ \hline
$\{1,2,5,8\}$&$\{3,4,6,7\}$&$\{1,2,7,8\}$&$\{3,4,5,6\}$&$\{1,3,4,6\}$&$\{2,5,7,8\}$\\ \hline
$\shortstack{(1,2,8)\\(2,5,8)}$&$\shortstack{(3,4,6)\\(4,6,7)}$&$\shortstack{(1,2,8)\\(7,8)}$
&$\shortstack{(3,4,6)\\(5,6)}$&$\shortstack{(1,3,4)\\(3,4,6)}$&$\shortstack{(2,5,8)\\(7,8)}$\\ \hline
\end{tabular}\vspace{10pt}
\begin{tabular}{c|c|c|c|c|c}\hline
$K_{I_6}$&$K_{\wh{I_6}}$&$K_{I_7}$&$K_{\wh{I_7}}$&$K_{I_8}$&$K_{\wh{I_8}}$\\ \hline
$\{1,3,4,7\}$&$\{2,5,6,8\}$&$\{1,4,6,7\}$&$\{2,3,5,8\}$&$\{1,4,7,8\}$&$\{2,3,5,6\}$\\ \hline
$\shortstack{(1,3,4)\\(1,4,7)}$&$\shortstack{(2,5,8)\\(5,6)}$&$\shortstack{(1,4,7)\\(4,6,7)}$
&$\shortstack{(2,3,5)\\(2,5,8)}$&$\shortstack{(1,4,7)\\(7,8)}$&$\shortstack{(2,3,5)\\(5,6)}$\\ \hline
\end{tabular}\vspace{5pt}\caption{Non-contractible full subcomplexes of $K$ with four vertices.}\label{tab:1}
\end{table}

\item\label{case:4} $|I|=5$. We need to use the following well known fact: Let $\Gamma$ be a simplicial complex on [m],
$\Gamma_J$ a full subcomplex on $J\subseteq [m]$. Then $\Gamma_{\wh{J}}$ is a deformation retract of $\Gamma\setminus\Gamma_J$.
From this and Alexander duality on $K$ we have that $\w {H}^j(K_I)\cong \w {H}_{2-j}(K_{\wh{I}})$. Since $|I|=5$, $|\wh I|=3$.
From the arguments in case \eqref{case:2}, $H_*(K_{\wh{I}})$ are all torsion free, so $\w H^*(K_{\wh I})\cong \w H_*(K_{\wh I})$.
Thus $\w H^*(K_I)$ is non-trivial if and only if $\wh I$ is one of the eight missing faces with three vertice,
and if so, $\w {H}^*(K_I)\cong \w {H}^1(K_I)\cong \mathbb{Z}$, whose generator we denote by $\beta_i$ ($1\leq i\leq 8$).

\item\label{case:5} $|I|=6$.
The same argument as in \eqref{case:4} shows that $\w H^*(K_I)$ is non-trivial if and only if $\wh I$ is $(5,6)$ or $(7,8)$,
and if so, $\w {H}^*(K_I)\cong \w {H}^2(K_I)\cong \mathbb{Z}$. Denote by $\lambda_1$ (respectively $\lambda_2$)
a generator of $\w {H}^2(K_I)$ for $\wh I=(5,6)$ (respectively $(7,8)$).

\item $|I|\geq 7$. If $|I|=7$, $\w H^*(K_I)=0$ is clear.  If $|I|=8$, $K_I=K$, so $\w H^*(K)\cong H^3(K)\cong\mathbb{Z}$. Denote by
$\xi$ a generator of it.
\end{enumerate}

\begin{proof}[Proof of Proposition \ref{prop:1}]
Theorem \ref{thm:1} and the preceding arguments give the cohomology group of $\mathcal {Z}_{K}$
\begin{center}
\renewcommand{\arraystretch}{1.5}
\begin{tabular}{|c|c|}\hline
$i$&$\w H^i(\mathcal {Z}_{K})\cong$\\ \hline
$1,2,4,8,10,11$&$0$\\ \hline
$3$&$\mathbb{Z}\cdot\psi(a_1)\oplus\mathbb{Z}\cdot\psi(a_2)$\\ \hline
$5$&$\bigoplus_{1\leq i\leq 8}\mathbb{Z}\cdot\psi(b_i)$\\ \hline
$6$&$\bigoplus_{0\leq i\leq 8}\big(\mathbb{Z}\cdot\psi(\alpha_i)\oplus \mathbb{Z}\cdot\psi(\alpha_i')\big)$\\ \hline
$7$&$\bigoplus_{1\leq i\leq 8}\mathbb{Z}\cdot\psi(\beta_i)$\\ \hline
$9$&$\mathbb{Z}\cdot\psi(\lambda_1)\oplus\mathbb{Z}\cdot\psi(\lambda_2)$\\ \hline
$12$&$\mathbb{Z}\cdot\psi(\xi)$\\ \hline
\end{tabular}
\end{center}

Now we give the cup product structure of $H^*(\mathcal {Z}_{K})$. First by Poincar\'e duality on $\mathcal {Z}_{K}$ and Remark \ref{rem:1}, up to sign
\begin{align}
\psi(a_i)\smile \psi(\lambda_i)&=\psi(\xi),\quad i=1,2;\label{eq:2}\\
\psi(b_i)\smile \psi(\beta_i)&=\psi(\xi),\quad i=1,\dots,8;\\
\psi(\alpha_i')\smile \psi(\alpha_i)&=\psi(\xi),\quad i=0,\dots,8.
\end{align}
Note that $K_{\wh {I_0}}=K_{(5,6)}*K_{(7,8)}$,
so up to sign $\psi(a_1)\smile \psi(a_2)=\psi(\alpha_0')$ (see Remark \ref{rem:1}), and so
\begin{equation}\label{eq:1}
\psi(a_1)\smile \psi(a_2)\smile \psi(\alpha_0)=\psi(\xi)
\end{equation}
Since $a_2*\alpha_0\in\w H^*(K_{\wh{(5,6)}})$,
$\psi(a_2)\smile \psi(\alpha_0)=\psi(a_2*\alpha_0)=p\cdot\psi(\lambda_1)$ for some $p\in \mathbb{Z}$.
From fomulae \eqref{eq:2} and \eqref{eq:1} we have $p=1$. Similarly,
$\psi(a_1)\smile \psi(\alpha_0)=-\psi(\lambda_2)$. Moreover from the arguments in case \eqref{case:5}, we have that
$\psi(a_i)\smile \psi(\alpha_j)=0$ for $1\leq j\leq 8$, and $\psi(a_i)\smile \psi(\alpha_j')=0$ for $0\leq j\leq 8$; $i=1,2$.
By an observation on the dimension of the non-trivial cohomology groups of $\mathcal {Z}_{K}$,
it is easy to verify that any other products between these generators are trivial.
Combining all the product relations above we get the desired result.
\end{proof}
There are other two different polytopal spheres from $K$
(corresponding to the two $4$-polytopes $P_{27}^8$ and $P_{29}^8$ in \cite{GS67}), so that
the corresponding moment-angle manifolds have the same cohomology rings as $\mathcal {Z}_{K}$.
The proof of this is the same as Proposition \ref{prop:1}.

For a moment-angle manifold corresponding to a simplicial $2$-sphere, if its cohomology ring is isomorphic to the one of a connected sum of
sphere products, then it is actually diffeomorphic to this connected sum of sphere products (\cite{BM06}, Proposition 11.6).
This leads to the following conjecture:
\begin{conj}
$\mathcal {Z}_{K}$ is diffeomorphic to the connected sum of sphere products in Proposition \ref{prop:1}.
\end{conj}
Note that the connected sum of sphere products in Proposition \ref{prop:1} only has one product of three spheres, we then ask:
Is there a moment-angle manifold (corresponding to a simplicial $3$-sphere) whose cohomology ring is isomorphic to the one of
a connected sum of sphere products with more than one product of three spheres? The following Theorem gives a negative answer to this question.
\begin{thm}\label{thm:a}
Let $K$ be a $n$-dimensional simplicial sphere ($n\geq 2$) satisfies $H^*(\mathcal {Z}_{K})\cong H^*(M)$,
where $M\cong M_1\#\cdots\#M_k$, and each $M_i$ is a product of spheres. Let $q_i$ be the number of sphere factors of $M_i$. Then
\begin{enumerate}[(a)]
\item If $q_i=n+1$ for some $i$, then $k=1$, and $\mathcal {Z}_{K}\cong M\cong S^3\times S^3\cdots\times S^3$.
\item Let $I=\{i:q_i\geq [\frac{n}{2}]+2\}$ (where $[\cdot]$ denotes the integer part). Then $|I|\leq 1$.
\end{enumerate}
\end{thm}

\begin{lem}\label{lem:1}
Let $K$ be a simplicial complex on $[m]$. Given two classes $[a],\,[b]\in \mathcal {H}^{0,*}(K)$, if $[a]*[b]\neq 0$, then there must
be a full subcomplex $K_I$ ($|I|\geq 4$) which is isomorphic to the boundary of a polygon, and satisfying $[a]*[b]\in \w H^1(K_I)$.
\end{lem}
\begin{proof}
Let $\mathcal {M}=\{I\in MF(K):|I|\neq 3\}$, and let $K'$ be a simplicial complex on $[m]$ so that $MF(K')=\mathcal {M}$.
Clearly, $K$ is a subcomplex of $K'$. Note that $K'$ and $K$ have the same $1$-skeleton, so if we can prove that for some $I\in [m]$,
$K'_I$ is isomorphic to the boundary of a polygon ($K'_I$ can not be the boundary of a triangle by the definition of $\mathcal {M}$),
then the result holds. From Remark \ref{rem:2}, there is a ring homomorphism $i^*:\mathcal {H}^{*,*}(K')\to \mathcal {H}^{*,*}(K)$
induced by the simplicial inclusion $i:K\hookrightarrow K'$.
It is easy to see that $i^*$ is a isomorphism when restricted to $\mathcal {H}^{0,*}(K')$.
Suppose $i^*([a'])=[a]$ and $i^*([b'])=[b]$. By assumption, $i^*([a']*[b'])\neq 0$, so $[a']*[b']\in \mathcal {H}^{1,*}(K')\neq 0$. Without loss of generality,
we can assume $[a']*[b']\in \w H^1(K'_J)$ for some $J\subset [m]$.
The lemma follows once we prove the following assertion:
\begin{ass}
For any simplicial complex $\Gamma$ satisfies $\w H^1(\Gamma)\neq 0$,
there must be a full subcomplex $\Gamma_I$ which is isomorphic to the boundary of a polygon, satisfying that
\[j^*:\w H^1(\Gamma)\to \w H^1(\Gamma_I)\]
is an epimorphism, where $j:\Gamma_I\to \Gamma$ is the inclusion map.
\end{ass}
Now we prove this. Since $H^1(\Gamma)\neq 0$, $H_1(\Gamma)\neq 0$, then there is a nonzero homology class $[c]\in H_1(\Gamma)$
represented by the $1$-cycle
\[c=(v_1,v_2)+(v_2,v_3)\cdots+(v_{k-1},v_k)+(v_{k},v_1),\]
where $v_i$ is a vertex of $\Gamma$. Without loss of generality, we assume $v_1,v_2,\dots,v_k$ are all different
and the vertex number $k$ is minimal among all $[c]$'s and their representations. Let $I=(v_1,\dots,v_k)$, we claim that $\Gamma_I$ is isomorphic to the
boundary of a polygon. If this is not true, then there must be a $1$-simplex, say $(v_1,v_j)\in \Gamma_I$ such that $j\neq 2,k$. Let
\[c_1=(v_1,v_2)+(v_2,v_3)\cdots+(v_j,v_1);\quad c_2=(v_1,v_j)+(v_j,v_{j+1})\cdots+(v_k,v_1).\]
Then $c=c_1+c_2$, and therefore $[c_1]\neq 0$ or $[c_2]\neq 0$. In either case, the vertex number of $c_i$ ($i=1,2$) is less than $k$, a contradiction.
Apparently, $j_*([c])$ is the fundamental class of $\Gamma_I$.
\end{proof}

\begin{lem}\label{lem:2}
Let $K$ be a simplicial sphere satisfies $H^*(\mathcal {Z}_{K})$ is isomorphic to the cohomology ring
of a connected sum of sphere products. If a proper
full subcomplex is isomorphic to the boundary of a $m$-gon, then $m\leq 4$.
\end{lem}

\begin{proof}
Suppose on the contrary that there is a proper full subcomplex $K_I$ isomorphic to the boundary of a $m$-gon with $m\geq 5$.
Then $H^*(\mathcal {Z}_{K_I})$ is a proper subring and a direct summand of $H^*(\mathcal {Z}_{K})$.
By Theorem \ref{thm:2} we can find five elements $a_1,a_2,b_1,b_2,c$ of $H^*(\mathcal {Z}_{K})$,
where $\mathrm{dim}(a_1)=3$, $\mathrm{dim}(a_2)=4$,
$\mathrm{dim}(b_1)=m-1$, $\mathrm{dim}(b_2)=m-2$ and $\mathrm{dim}(c)=m+2$, such that each of them is a generator of a $\mathbb{Z}$
summand of $H^*(\mathcal {Z}_{K})$, and the cup product relations between them are given by:
\[a_1\smile b_1=a_2\smile b_2=c,\]
all other products are zero. Clearly, $\mathrm{dim}(c)$ is not equal to the top dimension of $H^*(\mathcal {Z}_{K})$.
Suppose $H^*(\mathcal {Z}_{K})$ is isomorphic to the cohomology ring of
\[S^{f(1,1)}_{1,1}\times S^{f(1,2)}_{1,2}\times\cdots\times S^{f(1,k(1))}_{1,k(1)}\#\cdots\#
S^{f(n,1)}_{n,1}\times S^{f(n,2)}_{n,2}\times\cdots\times S^{f(n,k(n))}_{n,k(n)},\]
where $f$ is a function of $\mathbb{(Z^+)}^2\to\mathbb{Z^+}$ ($f(i,j)\geq3$ for all $i,j$), $S^{f(i,j)}_{i,j}=S^{f(i,j)}$,
$k(i)\in \mathbb{Z}^+$ denote the number of spheres in the $i$th summand of sphere product.
Denote by $e_{ij}^{(k)}$ a generator of $H^k(\mathcal {Z}_{K})$ corresponding to $S^k_{i,j}$ ($f(i,j)=k$). Then we can write
\[a_1=\sum_{f(i,j)=3}\lambda_{ij}e^{(3)}_{ij},\quad a_2=\sum_{f(i,j)=4}\lambda_{ij}'e^{(4)}_{ij}\]
where $\lambda_{ij},\,\lambda_{ij}'\in \mathbb{Z}$. It is easy to see that $e_{ij}^{(k)}\smile e^{(t)}_{rs}\neq0$
if and only if $i=r$ and $j\neq s$. Since $a_1\smile a_2=0$, we have that if $\lambda_{ij},\lambda_{i'j'}'\neq0$, then $i\neq i'$.
However this implies that $a_1\smile b_1\neq a_2\smile b_2$ since
$\mathrm{dim}(c)$ is not equal to the top dimension of $H^*(\mathcal {Z}_{K})$, a contradiction.
\end{proof}

\begin{lem}\label{lem:3}
Let $K$ be a $n$-dimensional simplicial sphere satisfies $H^*(\mathcal {Z}_{K})$ is isomorphic to the cohomology ring
of a connected sum of sphere products. If there is a
full subcomplex isomorphic to the boundary of a quadrangle, then for any full subcomplex $K_I$ saisfies $\w H^0(K_I)\neq0$,
we have $|I|=2$. Moreover, if $I_1,I_2$ are two different such sequences, then $K_{I_1\cup I_2}$ is isomorphic to the boundary of a quadrangle.
\end{lem}

\begin{proof}
The case $n=1$ are trivial, so we assume $n>1$. If we can prove the statement that for any two different missing faces $\sigma_1,\,\sigma_2\in MF(K)$, which contain two vertices, we have $\sigma_1\cap\sigma_2=\emptyset$, then the lemma holds.

Suppose $J=(1,2,3,4)$, and $MF(K_J)=\{(1,3),(2,4)\}$ (i.e., $K_J$ is isomorphic to the boundary of a quadrangle) by assumption.
First we will prove that for any vertex $v\not\in J$ and any $j\in J$, $(j,v)$ is a simplex of $K$.
Without loss of generality, suppose on the contrary that
$(1,5)\in MF(K)$. Let $\Gamma$ be a simplicial complex with vertex set $\{1,3,5\}$ such that $MF(\Gamma)=\{(1,3),(1,5)\}$.
Then $K_{(1,3,5)}$ is a subcomplex of $\Gamma$.
Clearly, $\w H^0(\Gamma)\cong \mathbb{Z}$, denote by $c_1$ a generator of it. Let $L=\Gamma*K_{(2,4)}$. Denote by $c_2$ a generator of
$\w H^0(K_{(2,4)})\cong \mathbb{Z}$, then an easy calculation shows that (see Remark \ref{rem:1}) $c_1*c_2$ is a generator of
$\w H^1(L)\cong \mathbb{Z}$. Let $J'=(1,2,3,4,5)$. Then $K_{J'}$ is a subcomplex of $L$, and the inclusion map induces
a monomorphism $\w H^1(L)\xrightarrow{\mu}\w H^1(K_{J'})$ (actually, $\mu(H^1(L))$ is a direct summand of $H^1(K_{J'})$).
There is a commutative diagram
\[\begin{CD}
\w H^0(\Gamma)\otimes\w H^0(K_{(2,4)})@>\eta >>\w H^1(L)\\
@V\phi\otimes\mathrm{id} VV @VV\mu V\\
\w H^0(K_{(1,3,5)})\otimes\w H^0(K_{(2,4)})@>\eta >>\w H^1(K_{J'}),
\end{CD}\]
where $\phi$ is induced by the inclusion map. So $\mu(c_1*c_2)=\phi(c_1)*c_2$ is a generator of $\w H^1(K_{J'})$.
Thus by Poincar\'e duality on $\mathcal {Z}_{K}$,
there is an element $c_0$ of $\w H^{n-2}(K_{\wh{J'}})$ such that $c_0*\phi(c_1)*c_2$ is a generator of
$\w H^n(K)\cong \mathbb{Z}$. On the other hand, let $e_1$ be a generator of $\w H^0(K_{(1,3)})$. Clearly $e_1*c_2$ is a generator
of $\w H^1(K_J)\cong \mathbb{Z}$, so there is a element  $e_0$ of $\w H^{n-2}(K_{\wh J})$ such that $e_0*e_1*c_2=c_0*\phi(c_1)*c_2$.
Since $e_0*e_1,\,c_0*\phi(c_1)\in \w H^{n-1}(K_{\wh{(2,4)}})\cong \mathbb{Z}$, we have $e_0*e_1=c_0*\phi(c_1)$. Since $\mathrm{dim}(\psi(e_1))=3$,
$\mathrm{dim}(\psi\phi(c_1))=4$, and $e_1*\phi(c_1)=0$,
then we get a contradiction by applying the arguments as in the proof of Lemma \ref{lem:2}.

Now suppose $v_1,v_2,v_3\in \wh J$ such that $(v_1,v_2),\,(v_1,v_3)\in MF(K)$. Let $J_0=(v_1,v_2,1,3)$. Then
from the result in the last paragraph we have $K_{J_0}$ is isomorphic to the boundary of a quadrangle. Thus by applying the
same arguments as in the last paragraph, we have that $(v_1,v_3)$ is a simplex of $K$, a contradiction.
\end{proof}
Now let us use the preceding results to complete the proof of Theorem \ref{thm:a}
\begin{proof}[Proof of Theorem \ref{thm:a}]
(a) From the assumption and Theorem \ref{thm:1}, we have that there are $n+1$ elements $c_i\in \w H^{k_i}(K_{J_i})$, $1\leq i\leq n+1$,
such that $\prod_{i=1}^{n+1} c_i\neq0\in \w H^n(K)$
(clearly, $J_i\cap J_k=\emptyset$ for $i\neq k$ and $\bigcup_{i=1}^{n+1}J_i$ is the vertex set of $K$).
From Remark \ref{rem:1}, the cohomology dimention of the class $\prod_{i=1}^{n+1} c_i$ is $n+\sum_{i=1}^{n+1}k_i$. Thus $k_i=0$ for all $1\leq i\leq n+1$.
Combine all the preceding lemmas, we have that $|J_i|=2$ for all $1\leq i\leq n+1$, so $J_i\in MF(K)$, and so $K$ is a subcomplex of
$K_{J_1}*\cdots *K_{J_{n+1}}$. Since $K_{J_1}*\cdots *K_{J_{n+1}}$ is a triangulation of $S^n$ itself,
then $K\cong K_{J_1}*\cdots *K_{J_{n+1}}$, and then the conclusion follows.

(b) Suppose there is a $M_u$ with $q_u\geq [\frac{n}{2}]+2$,
then as in (a) there are $q_u$ elements $c_i\in \w H^{k_i}(K_{J_i})$, $1\leq i\leq q_u$, such that $\prod_{i=1}^{q_u} c_i\neq0\in \w H^n(K)$.
The cohomology dimention of the class $\prod_{i=1}^{q_u} c_i$ is $q_u-1+\sum_{i=1}^{q_u}k_i$,
then from the inequality $q_u\geq [\frac{n}{2}]+2$, there are at least two $k_i$'s with $k_i=0$.
Then $K$ satisfies the conditions in all of the three Lemmas above.
From the first statement of Lemma \ref{lem:3}, we have that for any $a\in \mathcal {H}^{0,*}(K)$, $\mathrm{dim}(\psi(a))=3$.
So there are at least two $S^3$ factors in $M_u$.
From the second statement of Lemma \ref{lem:3}, we have that for any two linear independent element $a_1,a_2\in\mathcal {H}^{0,*}(K)$,
$a_1*a_2\neq 0$. This implies that all $S^3$ factors in the expression of $M$ are in $M_u$. Then there can not be another
$M_v$ with $q_v\geq [\frac{n}{2}]+2$. The conclusion holds.
\end{proof}

\providecommand{\bysame}{\leavevmode\hbox to3em{\hrulefill}\thinspace}
\providecommand{\MR}{\relax\ifhmode\unskip\space\fi MR }
\providecommand{\MRhref}[2]{%
  \href{http://www.ams.org/mathscinet-getitem?mr=#1}{#2}
}
\providecommand{\href}[2]{#2}

\end{document}